\documentclass[11pt]{amsart}
\usepackage[left=3.3cm,top=2.7cm,right=3.3cm]{geometry}             
\usepackage{graphicx}
\usepackage{amsthm, amssymb}
\usepackage{epstopdf}
\usepackage{mathtools}
\usepackage{array}
\usepackage[colorlinks=true, linkcolor=blue, anchorcolor=blue, citecolor=blue, filecolor=blue, menucolor= blue, urlcolor=blue]{hyperref}
\usepackage{tikz}
\usepackage{enumitem}
\usepackage{amsthm}
\theoremstyle{definition}
\newtheorem{lemma}{Lemma}
\newtheorem{prop}[lemma]{Proposition}
\newtheorem{theorem}[lemma]{Theorem}
\newtheorem{notation}[lemma]{Notation}
\newtheorem{definition}[lemma]{Definition}
\newtheorem{example}[lemma]{Example}
\newcommand{\pos}{\text{pos}}

\DeclareMathOperator{\init}{\mathcal{I}}
\DeclareMathOperator{\mun}{mun_G}
\DeclareMathOperator{\muneven}{mun_G^E}
\DeclareMathOperator{\munodd}{mun_G^O}
\DeclareMathOperator{\munevensp}{mun_{Q^n}^E}
\DeclareMathOperator{\munoddsp}{mun_{Q^n}^O}
\DeclareMathOperator{\munsp}{mun_{Q^n}}
\DeclareMathOperator{\munsupersp}{mun_{Q^4}}
\newcommand{\possr}{R_H}
\newcommand{\N}{\mathbb{N}}
\newcommand{\up}[1]{^{\uparrow #1}}
\DeclareMathOperator{\Hun}{Hun}
\DeclareMathOperator{\wt}{wt}
\DeclarePairedDelimiter\abs{\lvert}{\rvert} 
\DeclareMathOperator{\arrowsum}{sum}
\DeclareMathOperator{\diff}{diff}

\usepackage{todonotes}

\title{Hunting Rabbits on the Hypercube}
\author{Jessalyn Bolkema and Corbin Groothuis}
\thanks{Research partially supported by NSF-DMS grant \#1500662, ``The 2015 Rocky Mountain - Great Plains Graduate Research Workshop in Combinatorics". JB was supported by the US Department of Education GAANN Grant P200A120068.}

\begin{document}
\maketitle

\begin{abstract}

We explore the Hunters and Rabbits game on the hypercube. In the process, we find the solution for all classes of graphs with an isoperimetric nesting property and find the exact hunter number of $Q^n$ to be $1+\sum\limits_{i=0}^{n-2} \binom{i}{\lfloor i/2 \rfloor}$. In addition, we extend results to the situation where we allow the rabbit to not move between shots.
\end{abstract}
\section{Introduction and Background}
Consider a garden with tall bushes in which an invisible rabbit is hiding from a group of hunters who wish to shoot it. At each time step, the hunters, in unison, shoot darts at some set of bushes in the garden. If the rabbit is not hit by any darts, it is startled by the noise and jumps to an adjacent bush. This scenario forms the basis of a variant of Cops and Robbers \cite{Bonato} called Hunters and Rabbits \cite{Russoly}. The hunters win if they are guaranteed to hit the rabbit in a finite number of steps, while the rabbit wins if it can evade the hunters forever. 

This game can be modeled by a finite simple graph, where each bush is represented by a vertex and where edges exist between bushes that are within jumping distance for the rabbit. The rabbit starts at some vertex known only to itself, and the hunters have a vantage point from which they can choose to fire at any vertex. Given enough hunters, the rabbit is clearly doomed. Thus we ask: what is the minimum number of hunters that guarantees a successful hunt? This number is the hunter number of the graph, denoted $\Hun(G)$.

As the hunters get no feedback during the game, a hunter strategy is merely a predetermined list of bushes that the hunters will shoot at each time step. For a strategy to be a winning strategy, it must successfully hit a rabbit which starts anywhere and makes any legal sequence of choices. Because of this adversarial model, one can imagine the hunters handing a shot list to the rabbit, which it may use to look ahead to avoid getting stuck on a dangerous part of the graph. In this case, a winning strategy is a set of shots subject to which a rabbit with complete information still cannot avoid its fate.  

Consider as a basic example the path $P_n$ on vertex set $\{1,2,\ldots n\}$. First note that the strategy of shooting each vertex $1,2,\ldots n$ in order is successful if the rabbit started on an odd numbered vertex. Hence either $1,2,\ldots,n,1,2,\ldots,n$ (when $n$ is odd) 
or $1,2,\ldots,n,n,1,2,\ldots,n$ (when $n$ is even) is a strategy proving that $\Hun(P_n)=1$. In \cite{Britnell2013}, Britnell and Wildon found a necessary and sufficient condition for a graph $G$ to have $\Hun(G)=1$. 
\begin{theorem}[Britnell and Wildon, \cite{Britnell2013}]
$\Hun(G)=1$ if and only if $G$ is a tree and does not contain a subgraph isomorphic to the graph:

\begin{center}
\includegraphics[scale=0.75]{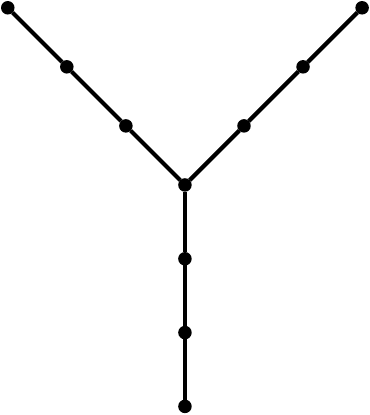}
\end{center}
\end{theorem}
Their setup of the problem used a single prince pursuing a princess, and their extension of the problem to multiple pursuers led to this game of hunters and rabbits.

Suppose that instead of a path, we had a cycle on $n$ vertices. Now at every step, the rabbit has two options available to it, so it can consult its list of hunter shots and choose the other one, forever evading a single hunter. However, a second hunter can choose a single vertex of the cycle to shoot at every time step. Now the rabbit is unable to pass through that vertex, leaving a path of length $n-1$ available as the remainder of the graph. As a path only requires one hunter, we see that two hunters can catch a rabbit on a cycle.

Other authors have considered the hunter number for specific classes of graphs. In \cite{Abramovskaya}, Abramovskaya, Fomin, Golovach, and Pilipczuk found that for an $n \times m$ grid graph G, $$\Hun(G)=\left \lfloor \dfrac{\min\{m,n\}}{2} \right \rfloor +1.$$ They also gave asymptotic lower and upper bounds for the hunter number of trees on $n$ vertices. In \cite{Gruslys}, Gruslys and M\'{e}roueh improve the bounds in \cite{Abramovskaya} to $\Theta(\log(n))$.

Our paper is organized in the following way. Section 2 gives the formal definitions of the Hunters and Rabbits game and develops basic tools and bounds used in the paper for general classes of graphs. In particular, Proposition \ref{prop:bound} of Section 2 is the main technical tool that gives rise to a tight lower bound for the hunter number of certain graphs in the later sections. In Section 3, we describe an optimal hunter strategy for classes of graphs that have the isoperimetric nesting property. A corollary of this result gives the hunter number for grids proved in \cite{Abramovskaya}. Section 4 considers a specific class of graphs with this property, hypercubes, and gives an exact calculation for the hunter number, using some classic results. Finally, in Section 5, we consider a variant where the rabbit is allowed to remain in its current position rather than move to an adjacent vertex after a shot.

Our notation is standard. For reference, see \cite{Bollobas}.

\section{Basic definitions and bounds for general graphs}

We begin by formalizing some of the notation for hunters and rabbits.

\begin{definition}[Hunter strategy]
A \emph{hunter strategy} on a graph $G$, denoted by $H=(H_1, H_2, \ldots, H_m)$, is a finite sequence of sets, where $H_i\subseteq V(G)$ for all $i$. We think of $H_i$ as the collection of vertices shot by hunters at time $i$. 
The \emph{concatenation} of two hunter strategies, $H_1=(H_1,H_2,\ldots, H_m)$ and $H_2=(H'_1, H'_2, \ldots, H'_n)$, is defined as 
\[ H_1\cdot H_2 = (H_1, H_2,\ldots,H_m,H'_1,H'_2,\ldots,H'_n).\] 

\end{definition} 
In order to keep track of the places the rabbit could be, given a specified hunter strategy, we introduce the following definition.

\begin{definition}
We denote the possible rabbit positions at time $i$ with respect to a hunter strategy $H$ by \[\possr(i+1)=N(\possr(i))\setminus H_{i}\subseteq V(G),\] with $\possr(1)=V(G)$. In addition, we define $r_H(i):=\left\vert \possr(i)\right\vert$.
\end{definition}

Notationally, this corresponds to the rabbit moving first at a time step, followed by the hunter shooting. This is just a shift of the ``natural'' time intervals where shots are fired first, but this notation is easier to use for the proofs. We note that in \cite{Abramovskaya}, they adopt the opposite ``time interval'' convention for their own purposes, but this change in frame does not change $\Hun(G)$.

Note that the possible rabbit positions depend on the hunter strategy; we are assuming that the rabbit has successfully evaded the hunters and therefore was not occupying any vertex in $H_i$ at time $i+1$. For this reason, the positions in $\possr(i)$ are known to the strategizing hunters. A walk $v_1 v_2\ldots v_k$ such that $v_i \in \possr(i)$ for all $i$ is called an \emph{H-rabbit walk}. We suppress $H$ when no confusion is possible.

\begin{definition}[Hunter number] 
Given a graph $G$, the \emph{hunter number} of $G$ is given by \[\Hun(G)=\min_{H \in \mathcal{H}(G)}\max_i \left \vert H_i \right \vert ,\] where $\mathcal{H}(G) = \{ H :$ there exists an $n \in \N$ such that $r_H(n)=0\}$
\end{definition}

The following lemma allows us to restrict our attention to connected graphs.

\begin{lemma}
$\Hun(G)$ is the maximum of $\Hun(C)$ over all components $C$ of $G$.
\end{lemma}
\begin{proof}
Since the hunters have a finite-time optimal strategy for each component, they can employ these strategies in sequence. 
\end{proof}
As one might expect, relationships between graph structures yield straightforward relationships between hunter numbers. We explore a few useful examples of this principle below. 
\begin{prop}
Let $G$, $G'$ be graphs and let $\varphi: G \rightarrow G'$ be a graph homomorphism. If $k=\max_{v\in V(G')}\lvert \varphi^{-1}(v)\rvert $, then $\Hun(G)$ $\leq k\Hun(G')$. In particular, if $G \subseteq G'$, $\Hun(G) \leq \Hun(G')$. 
\label{prop:hom}
\end{prop}
\begin{proof}
Let ${H'}=(H'_1,H'_2\ldots,H'_m$) be a hunter strategy for $G'$ requiring the minimum number of hunters. Define $H_i=\{v: \varphi(v) \in H'_i\}=\varphi^{-1}(H_i')$. Since $\varphi$ is a graph homomorphism, any rabbit walk in $G$ maps to a rabbit walk in $G'$. Since ${H'}$ is a winning hunter strategy for $G'$, all rabbit walks have finite length and thus each rabbit walk in $G$ must also be finite. Hence $H$ is a winning strategy for $G$ with $\max_i\left \vert H_i\right \vert\leq k\max_i\left \vert H_i'\right \vert$. 
\end{proof}
While this upper bound is useful for describing simple strategies for graphs with some inherent structure, the bound is often weak. For example, if we use the unique homomorphism from $K_{1,n}$ to $K_2$, this proposition would give $\Hun(K_{1,n})\leq n$, but it is easy to see that $\Hun(K_{1,n})=1$. We note that the special case of subgraphs was previously stated in \cite{Abramovskaya}.

Lower bounds arise from how sparse certain parts of the graph can be.
\begin{lemma}\label{lem:mindeg}
If $G$ is a graph with minimum degree $\delta$, then $\Hun(G)\geq\delta$.
\end{lemma}
\begin{proof}
Suppose there are $m < \delta$ hunters. At any time $t$, the rabbit has at least $\delta - m > 0$ neighbors that are not shot at time $t+1$, so the rabbit can evade the hunters forever.
\end{proof}

This bound can be improved by considering a graph locally.
\begin{lemma}
If the degeneracy of $G$ is at least $k$, then $\Hun(G)\geq k$.  
\end{lemma}
\begin{proof}
Let $H \subseteq G$ be a $k$-core of $G$. By Lemma \ref{lem:mindeg}, $\Hun(H) \geq k$ and by Proposition \ref{prop:hom}, $\Hun(G) \geq \Hun(H) \geq k$.
\end{proof}
Returning to the hunter strategy for paths of length $n$, a stronger lower bound on the hunter number of general graphs emerges. In that example, the hunters drive the rabbit toward the end of the path, where it is trapped. A natural way to catch any rabbit would be to force it into areas that require few hunters, narrowing the number of possible positions it could be at each time step. How many hunters are needed to possibly allow a decreasing sequence of possible rabbit positions?  

\begin{definition}
The \emph{minimum union of $k$ neighborhoods} for a graph $G$ is \[\mun(k):= \min\{\lvert N(W) \rvert: W \subseteq V(G), \lvert W \rvert = k\}.\] Note that $W\cap N(W)$ need not be empty.
\end{definition}

This definition is related to the study of isoperimetric inequalities. The classical isoperimetric inequality is a lower bound for the surface area of a region in $\mathbb{R}^n$ with a fixed volume. There are two primary types of isoperimetric inequalities in graphs. The \textit{edge boundary} for a vertex set $V$ is defined as $E(V,G\setminus V)$. The \textit{vertex boundary} for a vertex set $V$ is defined as $N(V)\setminus V$. We will be interested in the vertex boundary.

A vertex isoperimetric inequality for a graph is a lower bound on $N[V]$ for all sets $V$ of a fixed size. The function $\mun(k)$ replaces closed neighborhoods with open ones. While by definition, $\mun(k)$ will help give the lower bound, a closed form for the function is more desirable. We are often interested in defining the vertex sets that achieve these bounds. Harper in \cite{Harper} finds an isoperimetric inequality for the discrete hypercube and a vertex ordering achieving it. Leader and Bollob\'{a}s in \cite{Discretetorus} find an isoperimetric inequality for the discrete torus. For a survey on the topic, see \cite{Bezrukov}.

This minimum union of neighborhoods can be used to find a lower bound for the hunter number.
\begin{prop}\label{prop:bound}
Let $G$ be a graph and let $m_0= \max_k\{\mun(k)-k\}$. Then $\Hun(G)> m_0$.
\label{prop:mun}
\end{prop}
\begin{proof}
Let $k'$ be the smallest value of $k$ such that $\mun(k)-k=m_0$. We show by induction that $r_H(i) \geq k'$ for all $i$ if there are $m\leq m_0$ hunters. Note that $r_H(1) = \left \vert G \right \vert \geq k'$. Now 
\begin{align*}
r_H(i+1) &\geq \mun(r_H(i)) - m
\\&\geq \mun(k') - m
\\ &= k' + m_0 - m \geq k'.
\end{align*}
Since $r_H(i)\neq 0$ for any $i$, the rabbit is never caught.
\end{proof}
While this bound is fairly simple to state, it is not as easy to calculate. Let  $$u(G)=\max_k\{\mun(k)-k\}.$$ In section 4, we will identify this quantity for the hypercube. This bound is sometimes tight, a possibility we explore in the next section.

\section{Isoperimetric Nesting}
When it comes to bipartite graphs, there is additional power in the hunter strategies. We will always call the bipartition $E,O$ and say $E$ is the even part and $O$ is the odd part. This way, we will be able to talk about parities of vertices and sets of vertices. If the hunters knew in advance in which part of the graph the rabbit began, they could avoid wasting shots by only shooting at the part currently occupied by the rabbit. It turns out that if the hunters just pretend that they know the parity of the rabbit, that is sufficient. We will say that a strategy \textit{respects parity} if the hunters' shots alternate between parities and without loss of generality, a strategy respecting parity will start on the odd vertices.
\begin{lemma}[See for example \cite{Abramovskaya}]
Let $G$ be a bipartite graph with bipartition $E,O$. Any hunter strategy that respects parity and wins when the rabbit starts on an odd vertex can be extended to a strategy that wins for any starting rabbit position.
\label{lem:bip}
\end{lemma}

\begin{proof}
Suppose $H=(H_1,H_2,\ldots, H_n)$ wins when the rabbit starts in $O$. If $n$ is odd, then $H \cdot H$ is a winning strategy if the rabbit starts in either $E$ or $O$. If $n$ is even, then $H \cdot \emptyset \cdot H$ is a winning strategy if the rabbit starts in either $E$ or $O$.
\end{proof}

Because of this, it will be helpful to modify the idea of minimum unions of $k$ neighborhoods to restrict the collection of vertices to a single part. 

\begin{definition}[Minimum union of $k$ even or odd neighborhoods]
For a bipartite graph $G$, the \emph{minimum union of $k$ even neighborhoods} is \[\muneven(k):= \min\{\lvert N(V) \rvert: V \subseteq E, \lvert V \rvert = k\}\] and the \emph{minimum union of $k$ odd neighborhoods}
is \[\munodd(k):= \min\{\lvert N(V) \rvert: V \subseteq O, \lvert V \rvert = k\}.\] 

We then define analogously \[u^E(G)=\max_k\{\muneven(k)-k\}\] and \[u^O(G)=\max_k\{\munodd(k)-k\}.\]
\end{definition}
\begin{notation}
We denote by $\init(\leq,k)$ the initial segment of an order $\leq$ of size $k$.
\end{notation}

We now define \textit{isoperimetric nesting}. Essentially this says that there are total orders on $E$ and $O$ such that initial segments achieve $\mun$ and the neighborhood of an initial segment is an initial segment on the other side.
\begin{definition}[Isoperimetric nesting]\label{def:isonest}
We say that a bipartite graph $G$ has \textit{isoperimetric nesting} if there exists total orders $\leq_E$ on $E$ and $\leq_O$ on $O$ satisfying the following conditions for all $k$: 
\begin{enumerate}
\item $N(\init(\leq_E,k)) = \init(\leq_O, \muneven(k)),$
\item $N(\init(\leq_O,k)) = \init(\leq_E, \munodd(k)).$
\end{enumerate}

The orders $\leq_E$ and $\leq_O$ are together called the nest order. When the parity is clear, we will denote this order by $\leq$. 
\end{definition}

\begin{example}
The hypercube $Q^n$ with nest order defined by weightlex order on each part has isoperimetric nesting. We will define this order and prove this result in detail in section 4.
\end{example} 
 
We use isoperimetric nesting to generalize the proof strategy of \cite{Abramovskaya}, where the authors simultaneously prove that the grid has isoperimetric nesting and describe a strategy for it, though not in those terms. By disentangling these two ideas, we can develop a strategy for all graphs with isoperimetric nesting. 
\begin{theorem}\label{thm:main}
Let $G$ be a bipartite graph that has isoperimetric nesting. In addition, let $G$ satisfy the condition that $\vert u^E(G)-u^O(G) \vert \leq 1$. Then $\Hun(G)=m$, where $$m=\min\{u^E(G), u^O(G)\}+1.$$ 
\end{theorem}
\begin{proof}
Suppose without loss of generality that $m=u^E(G)+1$. By Lemma \ref{lem:bip}, we may also assume that the rabbit starts on the even side of the graph, i.e. that at any time $2k$, $\possr(2k)\subseteq E$. We generate a winning hunter strategy under that assumption iteratively by $H_i=\{$last $m$ nest-ordered vertices of $\possr(i)\}$. We show inductively that $$\possr(i)=
\begin{cases}
\init(\leq_E, r_H(i)), i=2k, k \in \mathbb{Z}\\
\init(\leq_O, r_H(i)), i=2k+1, k \in \mathbb{Z}\\
\end{cases}$$. By Definition \ref{def:isonest}, $N(\possr(i))=N(\init(\leq_E,r_H(i)))=\init(\leq_O,r_H(i))$, which is an initial segment. Now $$\possr(i+1) = N(\possr(i))\setminus H_i = \init(\leq, \vert N(\possr(i))\vert - m),$$ since $H$ removes the tail of the initial segment, leaving something that is still an initial segment. But $\vert N(\possr(i))\vert - m =r_H(i+1)$, completing the induction.

By definition of $u^E(G)$,
\begin{align*}
\mun(r_H(i))-r_H(i)&\leq u^E(G)=m-1 \\
\mun(r_H(i))-m &< r_H(i) \\
\vert N(\possr(i)) \setminus H_{i+1} \vert &< r_H(i) \\
r_H(i+1) &< r_H(i).
\end{align*}
Further, since we assumed $\vert u^E(G) - u^O(G) \vert \leq 1$, we have that
\begin{align*}
\mun(r_H(i+1))-r_H(i+1) &\leq m \\
\mun(r_H(i+1))-m &\leq r_H(i+1) \\
\vert N(\possr (i+1)) \setminus H_{i+2} \vert &\leq r_H(i+1) \\
r_H(i+2) &\leq r_H(i+1).
\end{align*}
Therefore $r_H(i+2) < r_H(i)$. Since $\{r_H(2k)\}$ is a decreasing sequence of natural numbers, eventually there is a $k$ such that $r_H(k)=0$, and so the rabbit is caught. This upper bound, along with the lower bound from Proposition \ref{prop:mun}, proves the result.
\end{proof}

While the requirement that $u^E(G)$ and $u^O(G)$ differ by at most one is necessary for this proof, there is no indication that it is necessary for the result. However, all the natural places to apply this theorem satisfy this condition. 
 
As noted above, applying Theorem \ref{thm:main} to $m\times n$ grids (with $m>n$) with the nest-ordering on the vertices given by $(x,y) \leq_{grid} (x',y')$ if $x+y < x'+y'$ or if $x+y=x'+y'$ and $x<x'$, matches the bound given by \cite{Abramovskaya}. In the next section, we show that hypercubes have isoperimetric nesting and that the nest order is, in fact, the weightlex order. We go on to find the corresponding hunter number explicitly.

\section{Hunters in the Hypercube}
We turn our attention to the discrete hypercube, $Q^n$.
\begin{definition}
The discrete hypercube $Q^n$ is the graph with $$V(Q^n)=\{0,1\}^n$$ $$E(Q^n)=\{xy: \exists ! \text{ i such that } x_i\neq y_i\},$$ 
This is isomorphic to the graph with
$$V(Q^n)=\mathcal{P}([n])$$
$$E(Q^n)=\{xy: \vert x \triangle y \vert = 1\}$$
where $\triangle$ is the symmetric difference.

Given a vertex $v \in Q^n$, we denote by $\wt(x)$ the size of the corresponding subset of $[n]$. If $\wt(x)$ is even, we call $x$ an even vertex, and similarly define odd vertices.
\end{definition}
We will talk about vertices as sequences or sets interchangeably. In what follows, we take particular advantage of the layered structure of $Q^n$, where all vertices of weight $k$ comprise the $k^{th}$ layer, for $0\leq k\leq n$. 
\begin{prop}
$\Hun(Q^n)\leq \binom{n}{\lfloor \frac{n}{2} \rfloor}$
\end{prop}
\begin{proof}
Apply Proposition~\ref{prop:hom} to the homomorphism $\phi: Q^n \rightarrow P_n$ defined by $\phi(x)=\vert x \vert$.
\end{proof}
The correct answer is reasonably close to this. We determine the correct answer by first proving that $Q^n$ has isoperimetric nesting using weightlex order. We now define this order.
\begin{definition}[Lexicographical order]
For $x,y \in Q^n$, we say $x$ is \emph{lexicographically less than} $y$, denoted $x \leq_{lex} y$, if $\min \{x\triangle y\} \in x $.
\end{definition}
\begin{example}
The lexicographical ordering on $Q^3$ is $$\{1,2,3\}<\{1,2\}<\{1,3\}<\{1\}<\{2,3\}<\{2\}<\{3\}<\emptyset .$$
\end{example}
\begin{definition}[Weightlex order] 
For vertices $x, y \in Q^n$, we say $x$ is \emph{weightlex less than} $y$, denoted $x\leq_{wlex} y$, if $|x| < |y|$ or $|x|= |y|$ and $x \leq_{lex} y$.
\end{definition}
\begin{example}
The weightlex ordering on $Q^3$ is
$$ \emptyset< \{1\}< \{2\}< \{3\}< \{1,2\}< \{1,3\}< \{2,3\}< \{1,2,3\}.$$
\end{example}
 Nest ordering needs to be defined separately on $E$ and $O$. We will use weightlex on each side, referred to as weightlex-even and weightlex-odd. When clear from context, we use $\preceq$ to denote both $\leq_{wlexeven}$ and $\leq_{wlexodd}$.
 \begin{example}
 The weightlex-even ordering on $Q^3$ is
 $$ \emptyset\prec \{1,2\}\prec \{1,3\}\prec \{2,3\}.$$
 The weightlex-odd ordering on $Q^3$ is
 $$  \{1\}\prec \{2\} \prec \{3\}\prec \{1,2,3\}.$$
 \end{example}
 In order to use Theorem \ref{thm:main}, we  show that the hypercube has isoperimetric nesting where the nesting order is weightlex order (respecting the partition). Harper showed in \cite{Harper} that weightlex order minimizes unions of closed neighborhoods in the hypercube. Korner and Wei proved the same thing for open neighborhoods in \cite{Korner}, but they did this by proving an analogue of the Kruskal-Katona form of Harper's original result.  We will show a direct proof here for open neighborhoods that follows the outline of Harper's original closed neighborhood proof. 
\begin{lemma}\label{lem:nlexislex}
Let $C$ be an initial segment of weightlex-even in $Q^n$. Then $N(C)$ is an initial segment of weightlex-odd. Similarly, if $C$ is an initial segment of weightlex-odd, then $N(C)$ is an initial segment of weightlex-even.
\end{lemma}
\begin{proof}
It is clear that the neighborhood of an entire even layer of $Q^n$ is all of the adjacent odd layers of $Q^n$, so it is sufficient to check only the upper neighborhood of an incomplete layer. There exists $x,y\in Q^n$ with $\wt(x)=\wt(y)=m$ while $y\in N(C)$ and $x\notin N(C)$. 
Suppose $x \prec y$. Let $i$ be the first position where $x_i < y_i$. If $i \neq m$, then $y \setminus \{y_m\}$ is the lowest neighbor of $y$ in $C$, and $y \setminus \{y_m\} \geq x \setminus \{x_m\} $, so since $C$ is an initial segment of weightlex-even, $x \setminus \{x_m\} \in C$. But then $x \in N(C)$. If $i = m$, then $y \setminus \{y_m\}= x \setminus \{x_m\} \in C$, so $x \in N(C)$. Therefore, $N(C)$ is an initial segment of weightlex-odd. 
Note that the same proof applies if we switch weightlex-even with weightlex-odd throughout.
\end{proof}
\begin{theorem} \label{ijcompress}
If $A \subset Q^n_{even}$ is a set of size $k$ and $C$ is the initial weightlex--even segment of size $k$, then $\lvert N(A) \rvert \geq\lvert N(C)\rvert $.
\end{theorem}
Our approach to proving this result will be to use compression to simplify the sets $A \subset Q^n_{even}$ that we need to check. 

\begin{definition}[$(i,j)$-Compression] Fix $\{i,j\} \subseteq [n]$. Let $A \subset Q^n_{even}$. From $A$, build auxiliary sets $A_{0,0}, A_{1,0}, A_{0,1}, A_{1,0} \in Q^{n-2}$, where 
\begin{eqnarray*}
 A_{0,0}& =& \{ x :x \in A, i \notin x, j \notin x\}, \\
 A_{0,1}& =& \{ x \setminus \{i\} :x \in A, i \in x, j \notin x\}, \\
 A_{1,0} &= &\{x \setminus \{j\}: x \in A, i \notin x, j \in x\}, \textrm{ and}\\
 A_{1,1} &=& \{x \setminus \{i,j\}: x \in A, i \in x, j \in x \}.
\end{eqnarray*}
Since each element of $A$ contributes exactly once to $\{A_{0,0}, A_{0,1}, A_{1,0}, A_{1,1}\}$, we can think of the sets $A_{0,0}, A_{0,1}, A_{1,0}, A_{1,1}$ as a decomposition of $A$.

The \emph{$(i,j)$-compression of $A$}, denoted $C_{i,j}(A)=B$ is the set $B$ such that 
\begin{eqnarray*}
 B_{0,0} & = & \init(<_{wlexeven},\vert A_{0,0} \vert), \\
 B_{0,1} & = &\init(<_{wlexodd},\vert A_{0,1} \vert), \\
 B_{1,0} & = & \init(<_{wlexodd},\vert A_{1,0} \vert), \textrm{ and} \\
 B_{1,1} & = & \init(<_{wlexeven},\vert A_{1,1} \vert).
\end{eqnarray*}

If $C_{i,j}(A)=A$, we say that $A$ is \emph{$(i,j)$-compressed}. If $A$ is \emph{$(i,j)$-compressed} for all $i,j$, then we say $A$ is \emph{compressed}. 
\end{definition}
\begin{lemma}
For any $\{i,j\}\subseteq [n]$, $C_{i,j}(A)$ satisfies $$\vert C_{i,j}(A) \vert = \vert A \vert \text{ and } \vert N(C_{i,j}(A))\vert \leq \vert N(A) \vert.$$
\label{lem:comp1}
\end{lemma}

\begin{proof}
Let $C_{i,j}(A)=B$. Since $\vert B_{\iota,\epsilon}\vert=\vert A_{\iota, \epsilon}\vert $ for $\iota, \epsilon \in \{0,1\}$, it is clear that $\vert B \vert = \vert A \vert$.\\ We can write
\begin{align*}
\vert N(A) \vert=&\vert(N(A_{0,0})\cup A_{1,0} \cup A_{0,1})\vert + \vert(N(A_{1,0})\cup A_{1,1} \cup A_{0,0})\vert
\\&+ \vert(N(A_{0,1})\cup A_{1,1} \cup A_{0,0})\vert + \vert(N(A_{1,1})\cup A_{1,0} \cup A_{0,1})\vert.
\end{align*}
and $N(B)$ similarly. Without loss of generality, let us look at the first of these terms.

We know that the neighborhood of an initial segment is also an initial segment, so $\vert N(B_{0,0}) \vert \leq \vert N(A_{0,0})\vert$ by induction on $n$. Because $B_{0,1}, B_{1,0}$, and $N(B_{0,0})$ are all initial segments of weightlex-odd order in $\mathcal{P}([n-2])$, they are nested in each other. So 
\begin{align*}
\lvert N(B_{0,0})\cup B_{1,0} \cup B_{0,1} \rvert &= \max\{ \lvert N(B_{0,0}) \rvert , \lvert B_{1,0} \rvert, \lvert B_{0,1}\rvert \}
\\ &= \max \{ \lvert N(B_{0,0}) \rvert , \lvert A_{1,0} \rvert, \lvert A_{0,1}\rvert\}
\\ &\leq \max \{ \lvert N(A_{0,0}) \rvert , \lvert A_{1,0} \rvert, \lvert A_{0,1}\rvert\}
\\ &= \lvert N(A_{0,0}) \cup A_{1,0} \cup  A_{0,1}\rvert.
\end{align*}

Therefore, $\lvert N(B) \rvert \leq \lvert N(A) \rvert$. 
\end{proof}
\begin{prop}
If $A \subset Q^n_{even}$, then there exists a sequence $(A_0, A_1, \ldots, A_k)$ for some $k$ with \begin{enumerate}
\item $\vert A_{i+1} \vert = \vert A_i \vert$ for all $i$,
\item $\vert N(A_{i+1})\vert \leq \vert N(A_i) \vert$ for all $i$, and
\item $A_k$ is compressed.
\end{enumerate}
\end{prop}
\begin{proof}
Let $A=A_0$. For $A_{l}$, if there is a pair $(i,j)$ for which $A_{l}$ is not $(i,j)$ compressed for a given $i,j$, set $A_{l+1}=C_{i,j}(A_l)$. At each step, $\sum_{v \in A_l} \pos(v)$ is strictly decreasing, where $\pos(v)$ denotes the position in weightlex ordering, so this process much terminate at some $A_k$ which is $(i,j)$-compressed for all $i,j$.
\end{proof}

In order to show that an $(i,j)$-compressed set is no better than an initial segment, we have a series of lemmas.
For the following lemmas, let $B$ be $(i,j)$-compressed for all $i,j$ but not be an initial segment of weightlex-even. In addition, let $C$ be the initial segment of weightlex-even order with size $\lvert B \rvert$.
\begin{lemma}\label{intandunion}
Suppose $B \subseteq Q^n_{even}$, where $B$ is compressed but not an initial segment of weightlex even. Let $x \in Q^n_{even} \setminus B$ and $y \in B$ with $x \prec y$. Then exactly one of the following is true:
\begin{enumerate}
\item $\lvert x \cap y\rvert =1$ and $x \cup y=\left[n\right]$
\item $\lvert x \cap y\rvert =0$ and $x \cup y=\left[n\right]\setminus\{k\}$ for some $k$
\item $x =y^c$.
\end{enumerate}
\end{lemma}
\begin{proof}
Note that if $\lvert x \cap y \rvert \geq 2$, then for a pair $i,j$ in their intersection, $C_{i,j}(B) \neq B$. This is a contradiction, since we assumed that $B$ is compressed. Therefore $\lvert x \cap y \rvert \leq 1$. 
If for any pair $i,j$, $\{i,j\} \nsubseteq y$ and $\{i,j\} \nsubseteq x$, then $C_{i,j}(B) \neq B$. Any pair then must intersect either $x$ or $y$, $\vert x \cap y \vert \geq n-1$. Therefore, the only possibilities are the three listed above.
\end{proof}
\begin{notation}
We denote the collection of $k$-sets of $[n]$ by $\binom{[n]}{k}$. The collection of sets of size at most $k$ in $[n]$ will be denoted $\binom{[n]}{\leq k}$.  
\end{notation}

\begin{lemma}\label{lem:evencompress}
Suppose $B \subseteq Q^n_{even}$ for $n$ even, where $B$ is compressed but not an initial segment of weightlex even. Then either 
\begin{enumerate}
\item $\displaystyle{B=\binom{\left[n\right]}{\leq \frac{n-2}{2}} \big\backslash  \left\{\frac{n+4}{2},\ldots,n\right\} \cup \left\{1,2,\ldots,\frac{n+2}{2}\right\} \text{ or }}$
\item $\displaystyle{B=\binom{\left[n\right]}{\leq \frac{n-2}{2}} \cup \left\{ x \in \binom{\left[n\right]}{\frac{n}{2}}: 1 \in x\right\} \setminus \left\{1,\frac{n+4}{2},\ldots, n\right\} \cup \left\{2,3,\ldots,\frac{n+2}{2}\right\}.}$ 
\end{enumerate}
Hence, if we let $C = \init(\preceq, \vert B \vert )$, $\lvert N(B) \rvert \geq \lvert N(C) \rvert$.
\end{lemma}
\begin{proof}
Let $x$ and $y$ be a pair satisfying the conditions of Lemma \ref{intandunion}. Considering the different possibilities in Lemma \ref{intandunion}, we can only be in case 3, because $\lvert x\rvert + \lvert y \rvert$ is even, and $n+1$ and $n-1$ are odd. So we must have that any $x \notin B$ and $y \in B$ with $x \prec y$ satisfies $x=y^c$. But then there is only one $x \notin B$ with $x \prec y$ for some $y \in B$. So $x$ and $y$ are adjacent in weightlex-even  order. There are only two adjacent pairs that are complements (which are given in the statement of the lemma). It is also simple to check in this case that $\lvert N(B) \rvert \geq \lvert N(C) \rvert$.
\end{proof}
\begin{lemma}\label{lemmay}
Suppose $B \subseteq Q^n_{even}$ for $n$ odd, where $B$ is compressed but not an initial segment of weightlex. Then for any pair $x,y$ as in Lemma \ref{intandunion}, we have one of the following
\begin{enumerate}
\item $\displaystyle{\vert x \vert ,\vert y\vert = \dfrac{n+1}{2}}$
\item $\displaystyle{\vert x \vert = \frac{n-1}{2}, \vert y \vert = \frac{n+3}{2}}$
\item $\displaystyle{\vert x \vert, \vert y \vert = \frac{n-1}{2}}$
\item $\displaystyle{ \vert x \vert = \frac{n-3}{2} , \vert y\vert = \frac{n+1}{2}.}$
\end{enumerate}
\end{lemma}
\begin{proof}
Since $n$ is odd, we can only be in cases 1 or 2 of Lemma \ref{intandunion}, since $\lvert x \rvert$ and $\lvert x^c \rvert$ have different parities. Let $y_{\max}=\max_{wlex}\{y: y \in B\}$ and $x_{\min}=\min_{wlex}\{x: x \notin B \}$.  Note that $\vert \{x \notin B:x \preceq y_{\max}\}\vert \leq \lvert y_{\max} \rvert + (n- \lvert y_{\max} \rvert)= n$. We also have that $\vert \{y \in B:y \succeq x_{\min}\}\vert \leq n.$ Therefore there are at most $2n-2$ elements between $x_{\min}$ and $y_{\max}$. But then $x_{\min}$ and $y_{\max}$ exist in layers at most two away from each other. Combining cases 1 and 2 from Lemma \ref{intandunion} with $x \in \binom{\left[n\right]}{m}$ and $y \in \binom{\left[n\right]}{m} \cup \binom{\left[n\right]}{m+2}$ gives the four options above.
\end{proof}
\begin{lemma}\label{lem:oddcompression}
If $B \subset Q^n_{even}$ for $n$ odd, where $B$ is compressed but not an initial segment of weightlex, then $\lvert N(B) \rvert \geq \lvert N(C) \rvert$, where $C = \init(\preceq, \vert B \vert )$.
\end{lemma}
\begin{proof}
We induct on the distance in weightlex-even between $x_{\min}$ and $y_{\max}$, as defined in the proof of Lemma \ref{lemmay}. We discuss cases according to the situations in Lemma \ref{lemmay}. Note that since for $y \prec x_{\min}, y \in B$, it is sufficient to only look at upper neighbors of $x_{\min}$.  
\begin{enumerate}[leftmargin=1 in]
\item[Case 1:] Enumerate elements of $y_{\max}$ and $x_{\min}$ as $y_{\max}=\{i_1,i_2,\ldots, i_{\frac{n+1}{2}}\}$ and $x_{\min}=\{j_1,j_2,\ldots, j_{\frac{n+1}{2}}\}$, with elements written in increasing order. Consider a set $x'= x_{\min} \cup \{i_l\}$ for $\{i_l\} \notin x_{\min} \cap y_{\max}$. Note that $x'$ is a neighbor of $y'= x_{\min} \setminus \{j_{\frac{n+1}{2}}\} \cup \{i_l\}$. Since $\lvert y_{max} \cap y' \rvert = 2$ and $y' \preceq y_{max}$, we must have that $y' \in B$. The choice of $i_l$ was arbitrary, so $N(x_{\min}) \subseteq N(B \setminus \{y_{\max}\})$. Therefore, $\lvert N(B \setminus \{y_{\max}\} \cup \{x_{min}\}) \rvert \leq \lvert N(B) \rvert$.
\item[Case 2:] This proof is similar to the previous case.
\item[Case 3:]  As before, say $y_{\max}=\{i_1,i_2,\ldots, i_{\frac{n-1}{2}}\}$ and $x_{\min}=\{j_1,j_2,\ldots, j_{\frac{n-1}{2}}\}$ with elements written in increasing order. Consider $x'= x_{\min} \cup \{i_\ell\}$ for some $1\leq \ell \leq \frac{n-1}{2}$. We have $x' \in N(y')$ for $y'= x_{\min} \setminus \{j_{\frac{n+1}{2}}\} \cup \{i_\ell\}$. Since $\lvert y_{\max} \cap y' \rvert = 2$ and $y' \preceq y_{\max}$, we must have that $y' \in B$.

Now consider $x_{\min} \cup \{k\}$ for $k$ as given in Lemma \ref{intandunion}. Suppose $j_l < k < j_{l+1}$. Then $x_{\min} \cup \{k\} \in N(y')$ for $y'=x_{\min} \cup \{k\}$. Since $y' \prec x_{min}$, we have $y' \in B$. Then $N(x_{\min}) \subseteq N(B \setminus \{y_{\max}\})$.

Else, $k> j_{\ell}$ for all $\ell$. If $k=n$, then $y \cup \{k\} \notin N(B \setminus \{y_{\max}\})$, so in the neighborhood count, $x_{\min} \cup \{k\}$ and $y_{\max} \cup \{k\}$ cancel out, so $N(x_{\min}) \subseteq N(B \setminus \{y_{\max}\})$.

Lastly, we consider if $k> j_{l}$ but $k\neq n$. If there exists a $y' \in B$ with $y'=x_{\min} \setminus \{j_l\}\cup \{k\} $, the result holds. Otherwise, let $$\mathcal{X}=\{x:x \notin B, x \preceq y_{\max}\}=\left\{x: x= x_{\min} \setminus \{j_l\} \cup \{k \}, 1 \leq l \leq \frac{n-1}{2} \right\}\cup x_{\min}.$$ Note that $N(\mathcal{X}) \setminus \{x_{\min} \cup \{k\}\} \subseteq N(B)$ and $\lvert \mathcal{X} \rvert=\frac{n+1}{2}$. 
\\
Now consider the set
\begin{align*}
\mathcal{Y}&=\{y: y \in B, y_{\max} \cup \{k\} \in N(y)\}
\\&=\left\{y: y \in B, y=y_{\max} \setminus \{i_l\} \cup \{k\}, 0 \leq l \leq \frac{n-1}{2}\right\}\cup y_{\max}
\end{align*}
Since $\lvert \mathcal{Y} \rvert \leq \lvert \mathcal{X} \rvert$, let us replace the lowest $\lvert \mathcal{Y} \rvert$ $x's$ in weightlex-even order in $B$ with $\mathcal{Y}$ to acquire $B'$. Since these two sets both have a vertex that is only in \emph{their} neighborhood, $\lvert N(B) \rvert \geq \lvert N(B') \rvert$.
\item[Case 4:] This proof is similar to the previous case. 
\end{enumerate}
\end{proof}

\begin{proof}[Proof of Theorem~\ref{ijcompress}]
The theorem follows from Lemmas \ref{lem:evencompress} and \ref{lem:oddcompression}.
\end{proof}
\begin{lemma} \label{munqn}
For $Q^n$, $\munevensp(k)= \munoddsp(k)$.
\end{lemma}
\begin{proof}
Let $\varphi:Q^n\rightarrow Q^n$ be the graph homomorphism that flips the first bit of a string. To show that this map preserves weightlex order on vertices with the same parity, suppose $x \leq_{wlex} y$ with $\lvert x \rvert \equiv \lvert y \rvert \pmod 2$. If $x$ and $y$ both have the same first bit, then flipping it preserves weightlex order. If $\abs{\wt(x) - \wt(y)} \geq 4$, then $\varphi$ preserves weightlex order. If $x$ has 0 in its first bit and $y$ does not and $\wt(x)=\wt(y)$, then $\wt(\varphi(x)) < \wt(\varphi(y))$, so $x \preceq y$.  If $x$ has 0 in its first bit and $y$ does not and $\wt(y)- \wt(x)=2$, then $\varphi(x)$ has 1 and $\varphi(y)$ does not, so $x \preceq y$. If $x$ has 1 in its first bit and $y$ does not, then $\wt(\varphi(x))<\wt(\varphi(y))$.

Since $\varphi$ is order preserving on same parity vertices and a graph homomorphism, it sends initial segments of weightlex-even to initial segments of weightlex-odd and vice versa, so $\munevensp(k)= \munoddsp(k)$.
\end{proof}
Combining these results, we have:
\begin{theorem}\label{thm:qn}
$Q^n$ has isoperimetric nesting with weightlex$_{even}$ and weightlex$_{odd}$ as the nesting orders, and therefore $\Hun(Q^n)=u^E(Q^n)+1$
\end{theorem}
\begin{proof}
By Theorem \ref{ijcompress}, the neighborhoods of initial segments achieve $\mun(k)$ and by Lemma \ref{lem:nlexislex}, the neighborhoods of initial segments are initial segments of the opposite parity. Thus $Q^n$ has isoperimetric nesting. By Lemma \ref{munqn}, we have that $u^O(Q^n)=u^E(Q^n)$, so we are able to apply Theorem \ref{thm:main} to get $\Hun(Q^n)$.
\end{proof}
\begin{example}
We give the hunter strategy prescribed by the theorem for $Q^4$. 
\begin{enumerate}
\item[$H_1:$] $\{1001, 0110, 0101, 0011, 1111\}$
\item[$H_2:$] $\{0010, 0001, 1110, 1101, 1011\}$
\item[$H_3:$] $\{1100, 1010, 1001, 0110, 0101\}$
\item[$H_4:$] $\{1000, 0100, 0010, 0001, 1110\}$.
\end{enumerate}

\end{example}

In order to calculate $u^E(Q^n)$, we look at the differences between successive values of $\muneven(k)$.
\begin{definition}[Difference sequence]
The \emph{difference sequence} of a graph $G$ is the sequence $(a_k)_{k=1}^{\lvert V(G) \rvert}$ where $a_k=\mun(k)-\mun(k-1)$.
\end{definition}
\begin{example}
The difference sequence for $Q^4$ is given by
(4,2,1,0,1,0,0,0,0).
\end{example}
Here, $u(Q^4)=5$. This is the maximum value of $\munsupersp(k)-k$, not $\munsupersp(k)$, but we can use the difference sequence to find either. Our first objective will be to give a description of the difference sequence of $Q^n$ for any $n$ by finding the difference subsequence for each layer. We will then use that sequence to find where the maximum values of $\mun(k)-k$ occur in order to find $u(G)$.


\begin{definition}
We define a sequence denoted $n\up{i}$ recursively by setting $n\up{0} = n$ and $n\up{i}= n\up{i-1}\cdot(n-1)\up{i}$ where $\cdot$ is the concatenation of sequences. We denote the length of $n\up{i}$ by $\lvert n\up{i} \rvert$ and the summation of all of the entries as $\arrowsum (n\up{i})$.
\end{definition}
\begin{example}\leavevmode
\begin{enumerate}
\item $4\up{1}=(4,3,2,1,0)$, with $\vert 4\up{1}\vert = 5$ and $\arrowsum(4\up{i})=10$.
\item $3\up{2}=(3,2,1,0,2,1,0,1,0,0)$, with $\vert 3\up{2}\vert = 10$ and $\arrowsum(3\up{2})=10$.
\item $2\up{3}=(2,1,0,1,0,0,1,0,0,0)$, with $\vert 2\up{3}\vert = 10$ and $\arrowsum(2\up{3})= 5$.
\end{enumerate}
\end{example}
\begin{lemma}
For $n,i$, we have $\lvert n\up{i} \rvert = \binom{n+i}{i}$ and $\arrowsum (n\up{i})=\binom{n+i}{i+1}$.
\end{lemma}
\begin{proof}
Using the recursive definition, we have that 
\begin{align*}
\lvert n\up{i} \rvert &= \lvert n\up{i-1} \rvert + \lvert (n-1)\up{i}\rvert 
\\ &= \binom{n+i-1}{i-1} + \binom{n+i-1}{i}
\\ &= \binom{n+i}{i}
\end{align*}
The computation for $\arrowsum (n\up{i})$ is similar.
\end{proof}
Note that $\vert n\up i \vert - \arrowsum (n\up i)$ is non-negative when $n \geq i+1$ and negative otherwise. If this difference is positive, we call this sequence a net gainer, and similarly define net loser and net neutral.
By Theorem \ref{thm:qn} we know that initial segments of weightlex-even achieve $\munevensp$, so there is a subsequence of our difference sequence corresponding to each even layer. 
\begin{example}
The difference subsequence of $Q^7$ corresponding to $\binom{[7]}{2}$ is $$(5,4,3,2,1,0,4,3,2,1,0,3,2,1,0,2,1,0,1,0,0).$$
\end{example}

\begin{prop}\label{prop:difference}
For $Q^n$, the difference subsequence corresponding to the $i$-th layer is given by $(n-i)\up{i}$ for $i \neq 1$. For $i=1$, the difference subsequence is $n\cdot(n-2)\up{1}$.
\label{prop:subsequence}
\end{prop}
First, we prove a lemma.
\begin{lemma}
For $Q^n$, if we denote the difference subsequence for the $i$-th layer as $\diff(n,i)$, we have that $\diff(n,i)=\diff(n-1,i-1)\cdot\diff(n-1,i)$ for $i\geq 2$.
\label{lem:subsequence}

\end{lemma}
\begin{proof}
Since we are looking at weightlex ordering, all of the sets containing 1 precede all the sets not including 0. If we ignore element 1 in the sets containing it, we are looking at subsets of $\{2,3,\ldots n\}$ of size $i-1$ in weightlex order, which generates exactly the sequence $\diff(n-1,i-1)$. If we ignore element 1 in sets not containing it, we are looking at subsets of $\{2,3,\ldots n\}$ of size $i$ in weightlex order, which generates exactly the sequence $(n-1,i)$. Since these match for all terms of the sequence besides the beginning and end, we only need note that the first entry of each is $n-i$ and the entry for $\{2,3,\ldots,i\}$ is $n-i-1$, so both sequences match at the ``connecting'' ends. 
\end{proof}
In proving both Lemma~\ref{lem:subsequence} and Proposition~\ref{prop:subsequence}, we see that the first vertex of weightlex-odd order is unusual in that it is the only vertex which has a neighbor in a lower layer which has not been covered before. If we ignored that lower neighbor, we could actually extend our argument to every layer besides the 0th one.
\begin{example}
The subsequence corresponding to $\binom{[5]}{1}$ in $Q^5$ is $(5,3,2,1,0)$ rather than the expected $(4,3,2,1,0)$. If we ignored the vertex corresponding to $00000$, however, we would get our expected sequence.
\end{example}
\begin{proof}[Proof of Proposition \ref{prop:difference}]
Using the notation as above, for $i\geq 3$, we can see that $$\diff(n,i)=\diff(n-1,i-1)\cdot\diff(n-1,i)=(n-1-(i-1))\up{i-1}\cdot(n-1-i)\up{i}=(n-i)\up{i}$$ by definition. For i=0, we have that the difference sequence is $n=n\up{0}$ as desired. Lastly for i=1, as per our discussion above, if we ignore the lower neighbor, we have $(n,1)=(n-1)\up{1}=(n-1)\up{0}\cdot(n-2)\up{1}$. Adding the lower neighbor back to the first entry, we get $n\cdot(n-2)\up{1}$.
\end{proof}
Now that we have a handle on the difference sequence for $Q^n$, we would like to get a grasp of where the local maxima of $\munsp(k)-k$ occur. 
\begin{prop}\label{maxmunup}
For $n>i$, if we decompose $n\up{i}$ as $n\up{i-1}(n-1)\up{i-1}\ldots 0\up{i-1}$, the last instance of the maximum value of $\munsp(k)-k$ for this subsequence occurs within $(i)\up{i-1}$. For $n\leq i$, the last instance of the maximum value of $\munsp(k)-k$ in this subsequence occurs in $n\up{i-1}$.
\end{prop}
\begin{proof}
By induction, we know that the last instance of the maximum value of each $k\up{i-1}$ occurs within $(i-1)\up{i-2}$ for any $k\geq i$. By the decomposition given in the definition, we have for $k\geq i$ that $k\up{i-1}=k\up{i-2}(k-1)\up{i-2}\ldots i\up{i-2} (i-1)\up{i-1}$. But since for $k\geq i$, $k\up{i-2}$ is a net gainer, each time we see the sequence $(i-1)\up{i-2}$ before $i\up{i-1}$, that sequence occurs higher than the previous one. Since $i\up{i-1}=i\up{i-2}(i-1)\up{i-2}$, and $i\up{i-2}$ is a net gainer, the maximum within $i\up{i-1}$ is higher than the maximum within $(i-1)\up{i-1}$. For $k\leq(i-1)$, $k\up{i-1}=k\up{k-1}(k-1)\up{k}(k-1)\up{k+1}\ldots(k-1)\up{i-1}$, by repeated use of the induction hypothesis, we find the maximum is in $k\up{k-1}$. But this subsequence occurs in $i\up{i-1}$ with net gainers preceding it, so it occurs lower in $k\up{i-1}$ than in $i\up{i-1}$. Since we have shown the local maxima all occur lower to the right and no higher to the left than in $i\up{i-1}$, the last maximum must appear in $i\up{i-1}$.

If $n\leq i$, we can decompose $n\up{i}=(n-1)\up{i-1}\ldots 0\up{i-1}$. Inductively, we get that the local maxima of each subsequence is $k\up{i-2}$. But each of these subsequences occur in $n\up{i-1}$ preceded by net losers, so each local maxima occurs lower than the previous. 
\end{proof}

\begin{lemma}\label{messlemma}
In $n\up{i}$, the position of a maximum value of $\munsp(k)-k$ occurs at position $$k=1+\sum\limits_{j=1}^{i-1}\binom{2j}{j-1}+\sum\limits_{j=i+1}^{n}\binom{j+i-1}{i-1}.$$ With value
$$\max=\biggl(\sum\limits_{j=1}^{i-1}\binom{2j}{j}+\sum\limits_{j=i+1}^{n}\binom{j+i-1}{i}\biggr)-k.$$
\end{lemma}
\begin{proof}
We decompose $n\up{i}=n\up{i-1}\ldots (i+1)\up{i-1}i\up{i-2}(i-1)\up{i-3}\ldots 2\up{0}1w$, where $w$ is the remainder of the sequence. We claim that the last maximum occurs at the 1 before the $w$. By repeatedly applying Proposition \ref{maxmunup}, we can iteratively find the subsequence containing the last maximum. This process terminates within $1\up{0}=1$, hence that $1$ is a max. Using the length and sum formulas for that index, we get the statement of the lemma.
\end{proof}

Using the above lemma, it is simple to calculate that in $Q^n$, the last maximum must occur in layer $\lceil n/2 \rceil$ . 
\begin{lemma}\label{lem:cumbersomeqn}
In $Q^n$, the maximum value of $\munsp(k)-k$ is 

$$\sum\limits_{i=0}^{\lceil n/4\rceil -1}\binom{n}{2i+1} - \sum\limits_{i=0}^{\lceil n/4\rceil -1}\binom{n}{2i}-\sum\limits_{i=1}^{\lfloor n/2 \rfloor-1}\binom{2i}{i-1} + \sum\limits_{i=1}^{\lfloor n/2 \rfloor - 1}\binom{2i}{i}, \text{for } n=0,3 \pmod 4$$
$$\sum\limits_{i=0}^{\lceil n/4 \rceil-1}\binom{n}{2i} - \sum\limits_{i=0}^{\lceil n/4\rceil -1}\binom{n}{2i-1}-\sum\limits_{i=1}^{\lfloor n/2 \rfloor-1}\binom{2i}{i-1} + \sum\limits_{i=1}^{\lfloor n/2 \rfloor - 1}\binom{2i}{i}, \text{for } n=1,2 \pmod 4.$$
\end{lemma}
\begin{proof}
The last two summations come from substituting $\lceil n/2 \rceil\up{\lceil n/2 \rceil-1}$ into Lemma \ref{messlemma}, which corresponds to the difference subsequence for layer $\lceil n/2 \rceil$ of $Q^n$. The first two summations come from considering the contributions to $\munsp(k)-k$ from the odd layers up to that point for ($n=1,2$ mod $4$) and the even layers (for $n=0,3$ mod $4$).
\end{proof}
This value is cumbersome, but turns out to collapse quite nicely.
\begin{theorem}\label{thm:qnmagic}
$$\Hun(Q^n)=1+\sum\limits_{i=0}^{n-2} \binom{i}{\lfloor i/2 \rfloor}.$$
\end{theorem}
\begin{proof}
We prove it by induction on $n$. We split these cases based on the value of $n \mod 4$. In both cases, we reduce to an alternating sum and repeatedly apply Pascal's identity at that point.

If $n\equiv 1 \pmod 4$, by Lemma \ref{lem:cumbersomeqn}, 
\begin{align*}
\Hun(Q^n)-1 - \Hun(Q^{n-1})-1=&~\sum\limits_{i=0}^{\lceil n/4\rceil -1}\binom{n}{2i} - \sum\limits_{i=0}^{\lceil n/4\rceil -1}\binom{n}{2i-1}-\sum\limits_{i=1}^{\lfloor n/2 \rfloor-1}\binom{2i}{i-1} 
\\ &+ \sum\limits_{i=1}^{\lfloor n/2 \rfloor - 1}\binom{2i}{i} - \sum\limits_{i=0}^{\lceil (n-1)/4\rceil -1}\binom{n-1}{2i+1} 
\\ &- \sum\limits_{i=0}^{\lceil (n-1)/4\rceil -1}\binom{n-1}{2i}-\sum\limits_{i=1}^{\lfloor (n-1)/2 \rfloor-1}\binom{2i}{i-1} 
\\ &+\sum\limits_{i=1}^{\lfloor (n-1)/2 \rfloor - 1}\binom{2i}{i}
\\ =&~  \binom{n}{0}+ \sum\limits_{i=0}^{\lceil n/4 \rceil-1} (-1)^{i+1}\binom{n}{i+1}+(-1)^{i}\binom{n-1}{i}
\\ =& \sum\limits_{i=0}^{\lceil n/4 \rceil-1}(-1)^{i}\binom{n-1}{i}
\\ =& \binom{n-2}{(n-2)/2}.
\end{align*}
If $n\equiv 2 \pmod 4$, by Lemma \ref{lem:cumbersomeqn}, 
\begin{align*}
\Hun(Q^n)-1 - \Hun(Q^{n-1})-1=~&\sum\limits_{i=0}^{\lceil n/4\rceil -1}\binom{n}{2i} - \sum\limits_{i=0}^{\lceil n/4\rceil -1}\binom{n}{2i-1}-\sum\limits_{i=1}^{\lfloor n/2 \rfloor-1}\binom{2i}{i-1}
\\ + \sum\limits_{i=1}^{\lfloor n/2 \rfloor - 1}&\binom{2i}{i} - \sum\limits_{i=0}^{\lceil (n-1)/4\rceil -1}\binom{n-1}{2i} - \sum\limits_{i=0}^{\lceil (n-1)/4\rceil -1}\binom{n-1}{2i-1}
\\ &-\sum\limits_{i=1}^{\lfloor (n-1)/2 \rfloor-1}\binom{2i}{i-1} + \sum\limits_{i=1}^{\lfloor (n-1)/2 \rfloor - 1}\binom{2i}{i}
\end{align*}
\begin{align*}
\\ =& \binom{n-2}{(n-2)/2} - \binom{n-2}{(n-4)/2} + \sum\limits_{i=0}^{\lceil n/2 \rceil -1}(-1)^{i}\left[\binom{n}{i}-\binom{n-1}{i}\right]
\\ =& \binom{n-2}{(n-2)/2} - \binom{n-2}{(n-4)/2}+ \sum\limits_{i=0}^{\lceil n/2 \rceil -2}(-1)^{i-1}\binom{n-1}{i}
\\ =& \binom{n-2}{(n-2)/2}+0
\\ =& \binom{n-2}{(n-2)/2}.
\end{align*}

The case of $n\equiv 3$ is similar to $n\equiv 1$, and the case of $n \equiv 0$ is similar to $n \equiv 2$.

\end{proof}

\section{Deaf Rabbit}
For practical purposes, it may be more interesting to consider the case where we allow the rabbit the flexibility to not move during a time step. This is equivalent to adding a loop to every vertex. We refer to this generalization as the hunter problem with a deaf rabbit, denoted Hun$[G]$. To revisit the earliest examples, we can no longer get away with one hunter for anything but the simplest of graphs, as the ``loop'' makes every non-isolated vertex at least degree 2. We can win with just 2 hunters for the path, by taking the basic strategy for the path and having the second hunter follow the strategy delayed by one time step. By adding a ``delayed'' hunter to the strategy for the cycle, we also have that $\Hun[C_n]=3$.

While we could attempt to establish a weak upper bound of $2\Hun(G)\geq \Hun\left[G\right]$ by using the time-delayed hunter modification, this will only work for suitable strategies. If an optimal strategy did not visit every vertex, for example, then a deaf rabbit could stay on an unvisited vertex forever. Even if an optimal strategy visits every vertex, we could conceivably not visit some vertex ``enough''. Amazingly, the strategies for the hypercube and similar graphs hold even in the deaf rabbit case.
\begin{definition}
Let $\mun[k]:= \min\{\lvert N[W] \rvert: W \subseteq V(G), \lvert W \rvert = k\}$.
\end{definition}
There is a natural extension of isoperimetric nesting in the closed neighborhood case as well.
\begin{definition}
We say that a graph has \it{closed isoperimetric nesting} if there is a chain of subsets $\mathcal{G}=(G_1,G_2,\ldots,G_n)$ satisfying the following conditions:
\begin{enumerate}
\item $\vert G_i \vert= i$.
\item $G_i \subset G_{i+1}$.
\item $\vert N[G_i] \vert = \mun[i]$. 
\item $N[G_i] = G_{\vert N[G_i]\vert}$.  
\end{enumerate}
\end{definition}
This leads to the analogous result to Theorem \ref{thm:main}:
\begin{theorem}
If G is a graph that has closed isoperimetric nesting, then $$\Hun[G]=1+\max_k\{\mun[k]-k\}.$$
\end{theorem}
\begin{proof}
The proof is identical to Theorem \ref{thm:main} with closed neighborhoods replacing open neighborhoods throughout. 
\end{proof}

Besides the initial vertex, the $\mun$-sequences for a layer are the same in the closed and open cases outside of layer 0 , since every vertex is covered by a lower layer. Thus we can find the value for $\Hun[Q^n]$.
\begin{prop}
$$\Hun[Q^n] = \binom{n}{\lceil n/2 \rceil} - \sum_{i=1}^{\lfloor n/2 \rfloor - 1} \binom{2i}{i-1} + \sum_{i=1}^{\lfloor n/2 \rfloor - 1} \binom{2i}{i}.$$
\end{prop}
\subsection*{Acknowledgement:} The second author is grateful to his advisors Jamie Radcliffe and John Meakin for providing direction on the problem. The authors would also like to thank the organizers of the 2015 Graduate Research Workshop in Combinatorics where the problem proposal originated and Brent McKain for presenting the problem. 

\bigskip
\bibliographystyle{plain}

\bibliography{myreferences}

\end{document}